\documentclass{amsart}

\usepackage{hyperref}
\usepackage{tikz}
\usepackage[mathcal]{euscript}
\usepackage{amsmath,amssymb,xy,mathrsfs,amsthm,amsxtra,graphicx,verbatim,upref}

\hypersetup{nesting=true,debug=true,naturalnames=true}
\hypersetup{colorlinks=true,urlcolor=magenta,citecolor=magenta,linkcolor=blue}

\newcommand{\Z}{\mathbf{Z}}
\newcommand{\F}{\mathbf{F}}
\newcommand{\Q}{\mathbf{Q}}
\newcommand{\T}{\operatorname{\mathsf{T}}}
\newcommand{\tors}{\mathsf{tors}}

\DeclareMathOperator{\height}{\rm ht}
\DeclareMathOperator{\typ}{{typ}}
\DeclareMathOperator{\I}{\mathtt{I}}
\DeclareMathOperator{\II}{\mathtt{II}}
\DeclareMathOperator{\III}{\mathtt{III}}
\DeclareMathOperator{\IV}{\mathtt{IV}}

\numberwithin{equation}{section}

\theoremstyle{plain}
\newtheorem{thm}[equation]{Theorem}
\newtheorem{lem}[equation]{Lemma}
\newtheorem{defn}[equation]{Definition}
\newtheorem{cor}[equation]{Corollary}
\newtheorem{prop}[equation]{Proposition}
\theoremstyle{remark}
\newtheorem{rmk}[equation]{Remark}

\newtheorem{exm}[equation]{Example}

 \xyoption{all}

\begin{document}

\title[Kodaira-Neron Statistics]{Kodaira-Neron statistics for rational elliptic curves with $j$-invariant 0 and 1728}

\author{John Cullinan}
\address{Department of Mathematics, Bard College, Annandale-On-Hudson, NY 12504, USA}
\email{cullinan@bard.edu}
\urladdr{\url{http://faculty.bard.edu/cullinan/}}

\author{Sebastian Sargenti}
\address{Department of Mathematics, Bard College, Annandale-On-Hudson, NY 12504, USA}
\email{ss5119@bard.edu}

\begin{abstract}
Elliptic curves over $\Q$ with $j$-invariant 0 or 1728 have additive reduction at all primes of bad reduction. In addition, all elliptic curves with $j$-invariant 0 have bad reduction at $p=3$ and all elliptic curves with $j$-invariant 1728 have bad reduction at $p=2$.   In this paper we count elliptic curves with $j$-invariant 0 and 1728 by height and determine asymptotics for the various Kodaira-N\'eron types at 3 and 2, respectively.  We also give related statistics by holding the torsion subgoup and isogeny-torsion graph constant. 
\end{abstract}



\maketitle

\section{Introduction}

\subsection{Motivation} Let $E$ be an elliptic curve over $\Q$.  If $p$ is a prime of bad reduction for $E$, then the nonsingular points of the reduction form an abelian group isomorphic to either $\F_p^\times$ or $\F_p$; we say the reduction type is multiplicative or additive, respectively.   The Kodaira-N\'eron type $\typ_p(E)$ of $E$ at $p$ gives a finer measure of the bad reduction by incorporating the size of the component group, as well as whether the reduction is potentially good or potentially multiplicative.  

Kodaira-N\'eron types have been studied for certain families of rational elliptic curves.  For example, fix a nontrivial torsion group $T$ and write $E_T$ for a parameterized family of minimal integral models of elliptic curves over $\Q$ with torsion subgroup $T$.  In \cite{br} the authors determined, for each prime $p$ of bad reduction, and for each $T/\Q$, the Kodaira-N\'eron type at $p$ as well as the local conductors and Tamagawa numbers for all curves in the $E_T$ family.  As an illustration of their results, we present the following example.

\begin{exm}
Let $T = \Z/5\Z$.  By \cite[Prop. 2.4, Lem. 2.5]{br} every minimal elliptic curve over $\Q$ with a 5-torsion point is a specialization of the model
\[
E(a,b): \ y^2 + (a-b)xy - a^2by = x^3 - abx^2,
\]
where $a$ and $b$ are coprime integers with $a$ positive.  If $E(a,b)$ has multiplicative reduction at $p$, then $\typ_p(E(a,b)) = \I_n$ with $n$ given explicitly in \cite[Thm.~2]{br_reps}.  By \cite[Thm.~3]{frey}, if $E$ has additive reduction at $p$, then $p=5$ (the converse is not necessarily true).   If $E(a,b)$ is additive at 5 then $v_5(a+18b) >0$ and by \cite[Table 19]{br} we have
\[
\typ_5(E(a,b)) = \begin{cases} \II & \text{ if $v_5(a+18b) =1$} \\ \III & \text{ if $v_5(a+18b) >1.$} \end{cases}
\]
\end{exm}

The main question that we  address in this paper is \emph{how often} each Kodaira-N\'eron type occurs within specified families of elliptic curves.  In order to make interesting comparisons, we choose families such that every member has a common prime of bad reduction.  In this paper we study elliptic curves over $\Q$ with $j$-invariant 0 and $j$-invariant 1728, which always have additive reduction at 3 and 2, respectively, and determine the exact proportion with which each Kodaira-N\'eron type occurs.   

\subsection{Setup and Notation} Every elliptic curve $E/\Q$ can be expressed in short Weierstrass form 
\[
E:~y^2 = x^3 + Ax+B,
\]
such that $A,B \in \Z$ with $4A^3+27B^2 \ne 0$. The $j$-invariant $j(E)$ of $E$ is given by 
\[
j(E) = 1728 \cdot \frac{4A^3}{4A^3 + 27B^2}
\]
and so $j(E) = 1728$ if $B =0$ and $j(E) = 0$ if $A=0$.   Recall that the \emph{height} $\height$ of $E$ is given by
\[
\height (E) = \max \lbrace 4|A|^3, 27B^2 \rbrace.
\]
By ordering all elliptic curves with specified $j$-invariant by height we can study the overall proportion with which each value of $\typ_3(E)$ and $\typ_2(E)$ occur for $j$-invariants 0 and 1728, respectively.  In the case of $j$-invariant 1728 we can further refine the statistics by controlling for isogeny-torsion graph, while if $j(E) = 0$ we can control for torsion subgroup.  

\begin{defn}
We write $N^0_{\T}(X)$ for the number of minimal elliptic curves of height $\leq X$ with $j$-invariant $0$ and Kodaira-N\'eron type $\T$ at 3.  Similarly, $N^{1728}_{\T}(X)$ denotes the number of minimal elliptic curves of height $\leq X$ with $j$-invariant $1728$ and Kodaira-N\'eron type $\T$ at 2.
\end{defn}

If $j(E) = 1728$, then $E$ belongs to an isogeny-torsion graph of isomorphism type $L_2(2)$ or $T_4^3$ (we will review this material below).  We can then refine $N_{\T}^{1728}(X)$ by holding the isogeny-torsion graph constant.

\begin{defn}
Let $G \in \lbrace L_2(2), T_4^3 \rbrace$. Then $N_{G,\T}^{1728}(X)$ denotes the number of minimal elliptic curves $E$ such that $j(E) = 1728$, $\typ_2(E) = \T$, $\height(E) \leq X$, and that belong to an isogeny-torsion graph isomorphic to $G$. 
\end{defn}

If $j(E) = 0$, then $E(\Q)_{\tors} \in \lbrace \lbrace \infty \rbrace, \Z/2\Z, \Z/3\Z, \Z/6\Z \rbrace$ (up to isomorphism, there is a unique elliptic curve with $j(E) = 0$ and $E(\Q)_\tors \simeq \Z/6\Z$).  In a similar fashion to the case of $j$-invariant 1728, we can refine $N_{\T}^{0}(X)$ by holding the torsion subgroup constant.

\begin{defn}
Let $G \in  \lbrace \lbrace \infty \rbrace, \Z/2\Z, \Z/3\Z \rbrace$. Then $N_{G,\T}^{0}(X)$ denotes the number of minimal elliptic curves $E$ such that $j(E) = 0$, $\typ_3(E) = \T$, $\height(E) \leq X$, and $E(\Q)_{\tors}  = G$.
\end{defn}

\subsection{Main Results} If $j(E) = 1728$, then $E$ belongs to one of four isogeny-torsion graphs: $T_4^1$, $T_4^2$, $T_4^3$, and $L_2(2)$.  Up to isomorphism, the graphs $T_4^1$ and $T_4^2$ only occur for the isogeny classes \href{https://beta.lmfdb.org/EllipticCurve/Q/32/a/}{32.a} and \href{https://beta.lmfdb.org/EllipticCurve/Q/64/a/}{64.a}, respectively, while $T_4^3$ and $L_2(2)$ occur infinitely often.  The graphs of type $T_4^3$ occur with different asymptotics than those of type $L_2(2)$, so we distinguish between them in the following theorem.

\begin{thm} \label{j1728_thm}
Let $E$ be an elliptic curve over $\Q$ with $j$-invariant 1728.  Let $\T = \typ_2(E)$.  
\begin{enumerate}
\item If $E$ belongs to an isogeny-torsion graph isomorphic to $T_4^3$, then \\
$\T \in \lbrace \II, \III, \I_2^*, \I_3^* \rbrace$ and 
\[
N_{T_4^3,\T}^{1728}(X) = \frac{c_{\T}}{\zeta(2)\sqrt[3]{2}}X^{1/6} + O(X^{1/12}),
\]
where
\[
c_{\T} = \begin{cases} 2/3 & \text{ if } \T =   \II \text{ or }\III, \\
1/3 & \text{ if } \T = \I_2^* \text{ or } \I_3^*. \end{cases}
\]
\item If $E$ belongs to an isogeny-torsion graph isomorphic to $L_2(2)$, then \\
$\T \in \lbrace \II, \III, \III^*, \I_2^*, \I_3^* \rbrace$ and 
and 
\begin{align*}
N_{L_2(2),\T}^{1728}(X) =  \frac{c_{\T}}{\zeta(4)\sqrt[3]{4}}X^{1/3} + O(X^{1/6}),
\end{align*}
where
\begin{align} \label{c_T_first_time}
c_{\T} = \begin{cases} 8/15 & \text{ if } \T = \III, \\
4/15 & \text{ if } \T = \II, \\
1/15 & \text{ if } \T = \I_2^*,  \I_3^*, \text{ or } \III^*. \end{cases}
\end{align}
\end{enumerate}
\end{thm} 

Because the number of graphs of type $L_2(2)$ dominate the number  of type $T_4^3$, the overall proportions of the Kodaira-N\'eron types for curves of $j$-invariant 1728 are the same as those that belong to graphs of type $L_2(2)$.  We record this in the following corollary.

\begin{cor}
Let $E/\Q$ be an elliptic curve with $j$-invariant 1728.  Then $\typ_2(E) \in \lbrace \I_2^*, \I_3^*, \II, \III, \III^* \rbrace$ and the proportion with which those types occur, among all minimal elliptic curves over $\Q$, are given by the $c_{\T}$ of (\ref{c_T_first_time}).
\end{cor}

If $j(E) = 0$, then each torsion subgroup $G \in \lbrace \lbrace \infty \rbrace, \Z/2\Z, \Z/3\Z \rbrace$ occurs with a different asymptotic.  We separate these results into distinct theorems for ease of reading.

\begin{thm} \label{j=0_tors=2}
Let $E$ be an elliptic curve over $\Q$ with $j$-invariant 0 and $E(\Q)_\tors \simeq \Z/2\Z$. Let $\T = \typ_3(E)$.   Then $\T \in \lbrace \III,\III^* \rbrace$ and 
\[
N_{\Z/2\Z,\T}^0(X) = c_{\T} \cdot \frac{2}{\sqrt{3}\zeta(2)} X^{1/6} + O(X^{1/12}),
\]
where 
\[
c_{\T} = \begin{cases} 3/4 & \text{ if $\T = \III$, and} \\ 1/4 & \text{ if $\T = \III^*$}. \end{cases}
\]
\end{thm}

\begin{thm} \label{j=0_tors=3}
Let $E$ be an elliptic curve over $\Q$ with $j$-invariant 0 and $E(\Q)_\tors \simeq \Z/3\Z$. Let $\T = \typ_3(E)$.   Then $\T \in \lbrace \II, \III, \IV, \IV^* \rbrace$ and 
\[
N_{\Z/3\Z,\T}^0(X) = c_{\T} \cdot \left( \frac{2\cdot3^{1/4}}{7\zeta(3)} \right) X^{1/4} + O(X^{1/12}),
\]
where 
\[
c_{\T} = 
\begin{cases} 
6/13 & \text{ if $\T = \II$}, \\
3/13 & \text{ if $\T = \III$ or $\IV$, and}  \\
1/13 & \text{ if $\T = \IV^*$.}
\end{cases}
\]
\end{thm}

\begin{thm}  \label{j=0_tors=0}
Let $E$ be an elliptic curve over $\Q$ with $j$-invariant 0 and $E(\Q)_\tors \simeq \lbrace \infty \rbrace$. Let $\T = \typ_3(E)$.   Then $\T \in \lbrace \II, \III, \IV, \II^*,\III^*,\IV^* \rbrace$, and
\[
N_{\lbrace \infty \rbrace ,\T}^0(X) = c_{\T} \cdot \frac{62}{63\sqrt{27}\zeta(6)} X^{1/2} + O(X^{1/4}),
\]
where 
\begin{align} \label{c_T_second_time}
c_{\T} = 
\begin{cases} 
243/364 & \text{ if $\T = \II$} \\
81/364& \text{ if $\T = \III$} \\
27/364& \text{ if $\T = \IV$} \\
9/364& \text{ if $\T = \IV^*$} \\
3/364& \text{ if $\T = \III^*$, and} \\
1/364& \text{ if $\T = \II^*$}.
\end{cases}
\end{align}
\end{thm}

\begin{cor}
Let $E/\Q$ be an elliptic curve with $j$-invariant 0.  Then $\typ_2(E) \in \lbrace \II, \III, \IV, \II^*, \III^{*},\IV^{*} \rbrace$ and the proportion with which those types  occur, among all minimal elliptic curves over $\Q$, are given by the $c_{\T}$ of (\ref{c_T_second_time}).
\end{cor}

\subsection{Organization of the Paper} In the next section we determine minimal models for elliptic curves with $j$-invariants 0 and 1728, including the additional conditions for the isogeny-torsion graphs and torsion subgroups.  We then apply Tate's algorithm to these minimal models to determine the precise arithmetic criteria for each Kodaira-N\'eron type.  Once we have determined these criteria, we apply standard sieve methods to determine the asymptotics.  We use the final section of the paper as a repository for most of the tables and figures.  Throughout the paper we use LMFDB \cite{lmfdb} notation for elliptic curves and isogeny classes and give links when appropriate.

\section{Weierstrass Models} \label{w_models_section}

\subsection{Minimal Models}

Every elliptic curve over $\Q$ can be written in short Weierstrass form 
\begin{align} \label{weier}
y^2 = x^3 + Ax +B
\end{align}
with $A, B \in \Z$.   Let $c_4$ and $c_6$ denote the usual invariants of $E$ and let $\Delta$ be the discriminant of $E$.  If $p$ is a prime number, then by \cite[Rmk.~1.1]{aec} if $v_p(\Delta) <12$, or $v_p(c_4)<4$, or $v_p(c_6)<6$, then $E$ is minimal at $p$.  If $p \ne 2,3$, then the converse holds as well, while if $p=2$ or $3$, then Tate's algorithm can be used to determine the minimality at $p$.  If $E$ is minimal at all primes $p$, then $E$ is said to be minimal. 

The Kodaira-N\'eron classification of singular fibers is an extensive subject and so we will only recall the topics we need and point the reader to \cite[Ch.~IV, \S9]{ataec} for a detailed background.  Tate's algorithm is an 11-step algorithm that determines, for each prime $p$, the local conductor, Tamagawa number, and Kodaira-N\'eron type.   (If the algorithm proceeds to Step 11, the elliptic curve is non-minimal and it is re-run under a suitable change of variables.)

If $p$ is multiplicative then $\typ_p(E) = \I_n$ for some $n \geq 1$, where $n = -v_p(j(E))$.  If $p$ is additive, then $p$ can be potentially multiplicative, in which case $\typ_p(E) = \I_n^*$ for $n \geq 1$, or potentially good.  If $p$ is potentially good, then $\typ_p(E) \in \lbrace \I_0^*, \II, \III, \IV, \II^*, \III^*, \IV^* \rbrace$.    If $E(\Q)_\tors$ is nontrivial, then the Kodaira-N\'eron types have been  determined by the work in \cite{br}. We now determine minimality conditions and the values of $\typ_2(E)$ and $\typ_3(E)$ when $j(E) = 1728$ or $0$, respectively.  Many of these cases have been determined in \cite{br} and we will give complete details for those that do not.

\begin{thm} \label{1728_min_thm}
Let $E$ be an elliptic curve over $\Q$ with $j$-invariant 1728 written in short Weierstrass form $y^2 = x^3 + Ax$ with $A \in \Z$. Then $E$ is minimal if and only if  $A$ is fourth-power-free.  Moreover, $\typ_2(E)$ is given in Tables \ref{j1728_finite_tables}, \ref{j1728_(2,2)tors},  \ref{j1728-Z2tors-size4}, and \ref{j1728-generic}. 
\end{thm}

\begin{proof}
We compute $\Delta = -64A^3$, $c_4 = -48A$, and $c_6=0$.  It follows immediately from the discussion at the start of this section that $E$ is minimal at primes $p\geq 5$ if and only if $A$ is fourth-power-free.  It remains to determine the minimality at $2$ and $3$.  All elliptic curves over $\Q$ with $j$-invariant 1728 have nontrivial torsion and therefore the Kodaira-N\'eron types are known by  \cite[Thms.~3.1, 3.4, 3.7]{br}.  We will not reproduce the details and only recap the arguments, as follows.

If $p=3$ and $v_3(A) = 1,2$, or 3, then Tate's algorithm terminates at Step 4, 6, or 9, with $\typ_2(E) = \III, \I_0^*$,  or $\III^*$, respectively. If $p=2$ and $v_2(A) = 1,2,$ or $3$, then Tate's algorithm terminates at Step 3, 4, 7, or 9, with $\typ_2(E) = \II, \III, \I_2^*$ or $\I_3^*$,  or $\III^*$, respectively.    If $v_2(A)$ or $v_3(A) \geq 4$ then Tate's algorithm proceeds to Step 11 and so $E$ is not minimal.  
\end{proof}

\begin{thm} \label{0_min_thm}
Let $E$ be an elliptic curve over $\Q$ with $j$-invariant 0 written in short Weierstrass form $y^2 = x^3 + B$ with $B \in \Z$. Then $E$ is minimal if and only if
\begin{enumerate}
\item $v_2(B) = 0,1,2,3,5$ and $B$ is sixth-power-free, or
\item $v_2(B) = 4$, $B/16 \equiv 3 \pmod{4}$, and $B$ is sixth-power-free.
\end{enumerate} 
Moreover, $\typ_3(E)$  is given in Tables \ref{j0_finite_tables}, \ref{j0_Z3_torsion}, \ref{j0_Z2_torsion}, and \ref{j0_trivial_torsion}.
\end{thm}

\begin{proof}
We compute  $\Delta = -2^43^3B^2$, $c_4 = 0$, and $c_6 = -2^53^3B$.  It follows from the discussion at the beginning of the section that if $p>3$ then $E$ is minimal if and only if $B$ is sixth-power free.  Let $p=2$ and suppose $v_2(B) >0$. If $v_2(B) = 1,2,3$,  or 5, then Tate's algorithm terminates at Step 3, 6 or 7, 7, or 10, to give $\typ_2(E) = \II$, $\I_0^*$ or $\I_4^*$, $\I_0^*$, or $\II^*$, respectively.   If $v_2(B) = 4$, then we proceed to Step 9.  Applying the standard translation of this step ($y \to y-2$), we consider the new Weierstrass equation for $E$:
\[
y^2 + 8y = x^3 + B - 16.
\]
Since $v_2(B) = 4$, it follows that $v_2(a_6) = v_2(B - 16) \geq 5$.  Since $a_4 =0$, we proceed to Step 10.  Note that $v_2(B-16) = 5$ if and only if $B/16 \equiv 3 \pmod{4}$. If so, then Tate's algorithm concludes at Step 10 with $\typ_2(E) = \II^*$.  If not, then the equation was not minimal to begin with. It now remains to determine the minimality and reduction type at 3.

If $v_3(B) = 0,1,2,3,4,$ or 5, then Tate's algorithm terminates at Step 2 or 3, 3, 5, 6 or 7, 8, or 10, respectively.  If $v_3(B)=6$, then $E$ is not minimal.  If $E$ has nontrivial torsion, $\typ_3(E)$ is given in \cite[Thms.~3.1, 3.4, 3.7]{br} and we do not reproduce those arguments here.  It remains to determine  $\typ_3(E)$ when $j(E) = 0$ and $E(\Q)_\tors$ is trivial.

If $E(\Q)_\tors$ is trivial, then $B$ is neither a cube nor a square (so that $E(\Q)[2]$ and $E(\Q)[3]$ are trivial, respectively).   Since $E$ has bad reduction at 3, we proceed to Step 2 of Tate's algorithm.  

Change variables to move the singular point to $(0,0)$ so that our new Weierstrass model is 
\begin{align} \label{step_2}
y^2 = (x-B)^3 + B = x^3 - 3Bx^2 + 3B^2x + (B-B^3).
\end{align}
Then $b_2 = -12B \equiv 0 \pmod{3}$ and so we proceed to Step 3.  We have $a_6 = B-B^3$ and thus $v_3(a_6) \geq 1$.  Step 3 of the algorithm concludes that $\typ_3(E) = \II$ if $v_3(a_6) = 1$, which we verify is equivalent to $b \equiv \pm 2, \pm 3, \pm 4 \pmod{9}$. We now proceed to Step 4.

In Step 4 we assume $v_2(a_6) \geq 2$ and compute $b_8 = 3B^4 - 12B^2$.  Then $v_2(b_8) \geq 2$.  We conclude that $\typ_3(E) = \III$ if and only if $v_3(b_8) =2$, which we verify is equivalent to $B \equiv \pm 1 \pmod{9}$.  We now proceed to Step 5.

We assume $v_3(b_8) \geq 3$ and compute $b_6 = 4B-4B^3$.  Then $\typ_3(E) = \IV$ if and only if $v_3(b_6)<3$, which is equivalent to $v_3(B) = 2$.  We now proceed to Step 6.

The cumulative results so far show that $v_3(B) \geq 3$. Therefore, $v_3(a_2) \geq 4$, $v_3(a_4) \geq 7$, and $v_3(a_6)\geq 3$, while $a_1=a_3=0$.   In Step 6 we consider the factorization of the polynomial
\[
P(T) = T^3 + a_{2,1}T^2+a_{4,2}T+a_{6,3} \equiv T^3 + a_{6,3} \pmod{3}
\]
Whether or not $a_{6,3} \equiv 0 \pmod{3}$, it is the case that $P(T)$ has a triple root modulo 3 and so we may skip Steps 6 and 7 and proceed to Step 8. 

Change variables  to move the triple root to 0.  Tate's algorithm terminates at Step 8 with $\typ_3(E) = \IV^*$ if and only if the polynomial
\begin{align} \label{ypoly}
Y^2 + a_{3,2}Y - a_{6,4}
\end{align}
has distinct roots, which occurs if and only if $v_3(B)=4$ or $v_3(B)=3$ and $B/27 \equiv \pm2,\pm4 \pmod{9}$.  We may now assume that the polynomial in (\ref{ypoly}) has a repeated root and proceed to Step 9.

Step 9 terminates with $\typ_3(E) = \III^*$ if $v_3(a_4) = 3$.  This occurs if and only if $v_3(B) = 3$ and $B/27 \equiv \pm  1 \pmod{9}$.  We proceed to Step 10, where $v_3(a_4) \geq 4$.

The cumulative effect so far is that $v_3(a_6) \geq 5$.  If $v_3(a_6) = 5$, then Step 10 terminates with $\typ_3(E) = \II^*$, which occurs if and only if $v_3(B) =5$.  If $v_3(a_6)>5$ then the equation is not minimal and we divide $B$ by $3^6$ and rerun the algorithm.  This completes the proof, and the results  are recorded in Table \ref{j0_trivial_torsion}. 
\end{proof}

\begin{rmk}
If $j(E) = 0$, $v_2(B) = 4$, and $B/16 \equiv 1 \pmod{4}$, then the substitution $x = 4X$ and $y = 8Y + 4$ yields the integral minimal model for $E$: 
\[
Y^2 + Y = X^3 + (B-16)/64.
\]
In addition, if $j(E) = 0$ and $E(\Q)_{\tors} \simeq \Z/3\Z$, our results in Table \ref{j0_Z3_torsion} results differ notationally from those of \cite[Table 10]{br}.  This is due to the different choice of model (they use $y^2  + ay = x^3$, whereas we use $y^2 = x^3 +B$). 
\end{rmk}

\subsection{Minimal Models With Torsion} We now  recap what is known about minimal models for elliptic curves with specified torsion when $j(E) \in \lbrace 0,1728 \rbrace$.

\subsubsection{Case 1: $j$-invariant 1728}  If $E/\Q$ is an elliptic curve with $j(E) = 1728$ that is given by the Weierstrass equation $y^2 = x^3 + Ax$, then $P := (0,0)$ is a rational point of order 2 on $E$.  By the main theorem of \cite{clz}, there are exactly four isogeny-torsion graphs over $\Q$ to which $E$ could belong, denoted $T_4^1$, $T_4^2$, $T_4^3$, and $L_2(2)$.  These graphs, augmented with the $j$-invariants of the curves, appear in Table \ref{1728_iso_tor_graphs}.  Up to isomorphism, the graphs $T_4^1$ and $T_4^2$ only occur for the isogeny classes \href{https://beta.lmfdb.org/EllipticCurve/Q/32/a/}{32.a} and \href{https://beta.lmfdb.org/EllipticCurve/Q/64/a/}{64.a}, respectively, while $T_4^3$ and $L_2(2)$ occur infinitely often. From now on we will restrict attention to these two cases.

\begin{prop} \label{1728_minimal_models}
Let $E/\Q$ be a minimal elliptic curve with $j$-invariant 1728.  Then $E$ is isomorphic to \href{https://beta.lmfdb.org/EllipticCurve/Q/32/a/3}{32.a3}, \href{https://beta.lmfdb.org/EllipticCurve/Q/32/a/4}{32.a4},  \href{https://beta.lmfdb.org/EllipticCurve/Q/64/a/3}{64.a3}, \href{https://beta.lmfdb.org/EllipticCurve/Q/64/a/4}{64.a4}, or is a nonzero integral specialization of one of the following curves:
\begin{center}
\begin{tabular}{|c|c|l|l|}
\hline
Graph & $E(\Q)_\tors$ & Model &  Conditions\\
\hline
$T_4^3$ & $\Z/2\Z \times \Z/2\Z$ & $y^2 = x^3-t^2x$ & $t$ squarefree\\
& $\Z/2\Z$ &$y^2 = x^3+t^2x$ & $t$ squarefree \\
\hline
$L_2(2)$ & $\Z/2\Z$ & $y^2 = x^3 +tx$ & $t \ne \pm s^2$, $t$ 4th-power-free \\
\hline
\end{tabular}
\end{center}
\end{prop}

\begin{proof}
The parameterizations appear in \cite[Table 5]{clz} and the minimality conditions follow immediately from Theorem \ref{1728_min_thm}.
\end{proof}

\subsubsection{Case 2: $j$-invariant 0} If $E/\Q$ is an elliptic curve with $j(E) = 0$, then $E(\Q)_\tors \in \lbrace \lbrace \infty \rbrace, \Z/2\Z, \Z/3\Z, \Z/6\Z \rbrace$.  Note that the group $\Z/6\Z$ only occurs in the isogeny class \href{https://beta.lmfdb.org/EllipticCurve/Q/36/a/}{36.a}.

\begin{prop} \label{j0_min_prop}
Let $E/\Q$ be a minimal elliptic curve with $j$-invariant 0.  Then $E$ is isomorphic to \href{https://beta.lmfdb.org/EllipticCurve/Q/27/a/3}{27.a3}, \href{https://beta.lmfdb.org/EllipticCurve/Q/27/a/4}{27.a4}, \href{https://beta.lmfdb.org/EllipticCurve/Q/36/a/3}{36.a3}, \href{https://beta.lmfdb.org/EllipticCurve/Q/36/a/4}{36.a4}, or is a nonzero integral specialization of one of the following curves:
\begin{center}
\begin{tabular}{|c|c|c|}
\hline
$E(\Q)_\tors$ & Model&  Conditions \\
\hline
$\Z/2\Z$ & $y^2 = x^3+t^3$ ($t \ne -3,1$)  & $t \ne s^2$, $t$ squarefree \\
\hline
$\Z/3\Z$ & $y^2 = x^3 + t^2$ ($t \ne 1, 4$) & $t \ne s^3$, $t$ cubefree \\
\hline
 & & $t \ne s^2, s^3$ \\
$\lbrace \infty \rbrace$& $y^2 = x^3 + t$ & $t$ 6th-power-free \\
&& if $v_2(t) = 4$, then $t/16 \equiv 3 \pmod{16}$ \\
\hline
\end{tabular}
\end{center}
\end{prop}

\begin{proof}
Similar to Proposition \ref{1728_minimal_models}, the parameterizations appear in \cite[Table 5]{clz} and the minimality conditions follow immediately from Theorem \ref{0_min_thm}.
\end{proof}

\section{The Proportion of Some Sets of Integers}

In this section we will determine the proportions of certain sets of integers with a view toward counting elliptic curves with specified Kodaira-N\'eron symbol.  Recall that the proportion of $k$th-power-free integers has proportion $1/\zeta(k)$ among all integers and
\[
\#\lbrace 1 \leq n \leq X \mid n \text{ is $k$th-power-free} \rbrace = \frac{x}{\zeta(k)} + O(X^{1/k}).
\]
We record the following lemma here for ease of reference later in the paper.

\begin{lem} \label{squarefree}
The proportion of $k$th-power-free positive integers $n$ such that:
\begin{enumerate}
\item $k=2$ and $n$ is odd is $(2/3) \cdot (1/\zeta(2))$; \label{odd_sqf}
\item $k=2$ and $n$ is coprime to 3, is $(3/4) \cdot (1/\zeta(2))$;
\item $k=3$ and $n$ is odd is $(4/7) \cdot (1/\zeta(3))$;
\item $k=3$ and $v_2(n)=1$ is $(2/7) \cdot (1/\zeta(3))$;
\item $k=4$ and $n$ is odd is $(8/15) \cdot (1/\zeta(4))$;
\end{enumerate} 
\end{lem}

These proportions are easily derived from the main theorem of \cite{integers_paper}.  However, we do not need the full strength of that result for Lemma \ref{squarefree}.  We need only augment certain local Euler factors of $\zeta(k)$.  For example, the proportion of Lemma \ref{squarefree}.\ref{odd_sqf} is given by 
\[
\frac{1}{\zeta(2)} \cdot \frac{\left(1-1/4 \right)^{-1}}{\left(1-1/2\right)^{-1}} = \frac{2}{3 \zeta(2)}
\]
since we remove the condition that $n$ be squarefree at 2 and replace it with the condition that $n$ be odd. The other proportions follow similarly. 

\section{The Proportion of Kodaira-N\'eron Symbols}

\subsection{$j$-invariant 1728} We now count minimal elliptic curves with $j$-invariant 1728 and their Kodaira-N\'eron symbols.  Recall that with two exceptions, all such elliptic curves have torsion subgroup $\Z/2\Z$ or $\Z/2\Z \times \Z/2\Z$.  The torsion subgroup $\Z/2\Z$ shows up in both the isogeny-torsion graph $T_4^3$ and $L_2(2)$.  Since the $L_2(2)$ graphs dominate the $T_4^3$, we choose to  filter our counts by isogeny-torsion graph, rather than by torsion subgroup.   We now prove Theorem \ref{j1728_thm}.

\begin{proof}[Proof of Theorem \ref{j1728_thm}]
Set $G = T_4^3$ and let $X >0$.  We first count minimal elliptic curves with torsion subgroup $\Z/2\Z \times \Z/2\Z$ and height $\leq X$.  By  Proposition \ref{1728_minimal_models}, these are given by 
\[
y^2 = x^3 -t^2 x,
\]
with $4t^6 <X$ and $t$ squarefree; the number of such integers is 
\[
\frac{1}{\zeta(2)} \cdot \left(\frac{X}{4}\right)^{1/6} + O(X^{1/12}).
\]
Of these elliptic curves, $\typ_2(E) = \III$ if $t$ is odd and $\I_2^*$ if $t$ is even, by Theorem \ref{1728_min_thm} and Table \ref{j1728_(2,2)tors}.  By Lemma \ref{squarefree} the proportion of odd squarefree integers is $2/3\zeta(2)$, while the proportion of even squarefree integers is $1/3\zeta(2)$.  Therefore, 
\[
N_{T_4^3,\T}^{1728}(X) =\frac{c_{\T}}{\zeta(2)\sqrt[3]{2}}X^{1/6} + O(X^{1/12}),
\]
where
\[
c_{\T} = \begin{cases} 2/3 & \text{ if } \T =   \III, \text{ and} \\
1/3 & \text{ if } \T = \I_2^*. \end{cases}
\]
It now remains to determine the asymptotic for the case of $j=1728$, $G = T_4^3$, and torsion subgroup $\Z/2\Z$.  By Proposition \ref{1728_min_thm}, all such minimal elliptic curves are given by 
\begin{align} \label{min_j1728_Z2_tors}
y^2 = x^3 + t^2x,
\end{align}
with $t$ squarefree; by Table  \ref{j1728-Z2tors-size4} they satisfy $\typ_2(E) = \II$ if $t$ is odd and squarefree and $\typ_2(E) = \I_3^*$ if $t$ is even and squarefree.   The height of an elliptic curve with Weierstrass model given by Equation (\ref{min_j1728_Z2_tors}) is $4t^6$.  Therefore, 
\[
N_{T_4^3,\T}^{1728}(X) =\frac{c_{\T}}{\zeta(2)\sqrt[3]{2}}X^{1/6} + O(X^{1/12}),
\]
where
\[
c_{\T} = \begin{cases} 2/3 & \text{ if } \T =   \II, \text{ and} \\
1/3 & \text{ if } \T = \I_3^*. \end{cases}
\]
This completes the proof of the case $G= T_4^3$.  

Next, let $G = L_2(2)$.  Working modulo $2^4$, it follows directly from the arithmetic criteria in Table \ref{j1728-generic} that Kodaira-N\'eron types $\III, \II, \I_2^*, \I_3^*$, and $\III^*$ occur with proportions $8/15$, $4/15$, $1/15$, $1/15$, and $1/15$, respectively.  The models of $j$-invariant 1728 elliptic curves that are minimal,  belong to the graph $L_2(2)$, and have height $\leq X$  are parameterized by fourth-power-free integers $t$ that are not squares such that $4|t|^3\leq X$.

The number of  fourth--power-free integers $t$ with $t \leq (X/4)^{1/3}$ is  
\[
\frac{1}{\zeta(4)} \cdot \left(X/4\right)^{1/3}  + O(X^{1/12}),
\]
while the number of square integers $t \leq  (X/4)^{1/3}$ is  $\sim \sqrt{(X/4)^{1/3}} = O(X^{1/6})$.  It follows that 
\[
N_{L_2(2),\T}^{1728}(X) = \frac{1}{\zeta(4)} \left(\frac{X}{4}\right)^{1/3} + O(X^{1/6}),
\]
where 
\[
c_{\T} = \begin{cases} 8/15 & \text{ if } \T = \III, \\
4/15 & \text{ if } \T = \II, \\
1/15 & \text{ if } \T = \I_2^*,  \I_3^*, \text{ or } \III^*,\end{cases}
\]
as claimed. 
\end{proof}

\begin{rmk}
Observe that the four values $2/3, 2/3, 1/3, 1/3$ of $c_{\T}$ in  Theorem \ref{j1728_thm}, part 1, do not sum to 1.  This is due to the fact that we are working with two distinct Weierstrass models in the course of the proof.  The \emph{proportions} with which the various Kodaira-N\'eron symbols occur are 1/3, 1/3, 1/6, and 1/6. 
\end{rmk}

\begin{cor}
Let $E/\Q$ be an elliptic curve with $j$-invariant 0.  Then $\typ_2(E) \in \lbrace \II, \III, \IV, \II^*, \III^{*},\IV^{*} \rbrace$ and the proportion with which those types  occur, among all minimal elliptic curves over $\Q$, are given by the $c_{\T}$ of Equation (\ref{c_T_second_time}).
\end{cor}

\begin{proof}
The total number of minimal elliptic curves of $j$-invariant 1728 and height $\leq X$ is $O(X^{1/3})$; those that belong to  isogeny-torsion graphs isomorphic to $L_2(2)$ contribute $O(X^{1/3})$, while those belonging to $T_4^3$ contribute $O(X^{1/6})$. The former dominate as $X \to \infty$, as do the proportions of Kodaira-N\'eron symbols from those graphs. 
\end{proof}

\begin{exm}
In Pari/GP \cite{pari} we count the number of minimal elliptic curves of $j$-invariant 1728 up to height $10^{18}$ and their respective Kodaira-N\'eron symbols as follows:
\begin{center}
\begin{tabular}{|lr|}
\hline
Symbol & Number \\
\hline
$\III$ & 620846 \\
$\II$ &  310424 \\
$\I_2^*$ & 77607 \\
$\I_3^*$ & 77607 \\
$\III^*$ & 77610\\
\hline
\end{tabular}
\end{center}
These agree with the asymptotics of Theorem \ref{j1728_thm}.
\end{exm}

\subsection{$j$-invariant 0} \label{j0-proportion-section}

Now we determine the proportions of Kodaira-N\'eron symbols at $p=3$ for elliptic curves with $j$-invariant 0 and various torsion subgroups.  We prove Theorems 
 \ref{j=0_tors=2}, \ref{j=0_tors=3}, and \ref{j=0_tors=0} in succession. 

\begin{proof}[Proof of Theorem \ref{j=0_tors=2}]
An elliptic curve with $j$-invariant 0 has torsion subgroup $\Z/2\Z$ if and only if $E$ can be written as $y^2 = x^3 + t^3$ with $t \ne 1$.  Applying Proposition \ref{j0_min_prop}, $E$ is minimal if and only if $t$ is squarefree. 
Therefore, the number of minimal elliptic curves with torsion subgroup $\Z/2\Z$ of height $\leq X$ is obtained by estimating the number of squarefree integers $t$ such that $27t^6 \leq X$; this is
\[
\frac{2}{\zeta(2)} \cdot \left(\frac{X}{27} \right)^{1/6} + O(X^{1/12})= \frac{2}{\sqrt{3}\zeta(2)} X^{1/6} + O(X^{1/12}).
\]
By Table \ref{j0_Z2_torsion}, we have $\typ_3(E) = \III$ if $v_3(t) = 0$, and $\typ_3(E) = \III^*$ if $v_3(t) = 1$.  By Lemma \ref{squarefree}, these valuations occur 3/4 of the time and 1/4 of the time, respectively, among all squarefree integers.  
\end{proof}

\begin{proof}[Proof of Theorem \ref{j=0_tors=3}]
An elliptic curve with $j$-invariant 0 has torsion subgroup $\Z/3\Z$ if and only if $E$ can be written as $y^2 = x^3 + t^2$ with $t \ne 1$. Observe that if $v_2(t^2) = 4$, then it is not possible for $t^2/16 \equiv 3 \pmod{4}$, since odd squares are congruent to 1 modulo 4.  Therefore, by Proposition \ref{j0_min_prop}, $E$ is minimal if and only if $t$ is cubefree and $v_2(t) \in \lbrace 0,1 \rbrace$.   By Lemma \ref{squarefree} the proportion of such integers is $6/7\zeta(3)$. Counting by height, the number of minimal elliptic curves of height $\leq X$ with $j$-invariant 0 and torsion subgroup $\Z/3\Z$ is therefore
\[
\frac{6}{7\zeta(3)} \left(\frac{X}{27}\right)^{1/4} + O(X^{1/12}) = \left( \frac{2\cdot3^{1/4}}{7\zeta(3)} \right) X^{1/4} + O(X^{1/12}).
\]
The proportions of the various Kodaira-N\'eron types follow directly from the congruences modulo 27 of Table \ref{j0_Z3_torsion}.  Therefore, 
\[
N_{\Z/3\Z,\T}^0(X) = c_{\T} \cdot \left( \frac{2\cdot3^{1/4}}{7\zeta(3)} \right) X^{1/4} + O(X^{1/12}),
\]
where 
\[
c_{\T} = 
\begin{cases} 
6/13 & \text{ if $\T = \II$}, \\
3/13 & \text{ if $\T = \III$ or $\IV$, and}  \\
1/13 & \text{ if $\T = \IV^*$.}
\end{cases}
\]
\end{proof}

\begin{proof}[Proof of Theorem \ref{j=0_tors=0}]
An elliptic curve with $j$-invariant 0 has trivial torsion subgroup if and only if $E$ can be written as $y^2 = x^3 + t$ with $t$ neither a square nor a cube.  The number of such $t$ such that $27t^2 \leq X$ is $\sim (X/27)^{1/2} - (X/27)^{1/4} - (X/27)^{1/3} + (X/27)^{1/6}$.  The conditions on $t$ for minimality (6th-power-free, except that if $v_2(t) = 4$, then $t/16 \equiv 3 \pmod{16}$), are sieved by multiplying by $31/32\zeta(6)$, where we replace the Euler factor at 2 in $\zeta(6)$ (64/63) by 64/62.  Therefore there are 
\[
\frac{62}{63 \sqrt{27}\zeta(6)} X^{1/2} + O(X^{1/4}) 
\]
minimal elliptic curves of $j$-invariant 0 and trivial torsion of height $\leq X$.  The proportions with which each Kodaira-N\'eron type occurs follows immediately from working with congruences modulo $3^6$ using the criteria of Table \ref{j0_trivial_torsion}.  We record these proportions as
\[
N_{\lbrace \infty \rbrace ,\T}^0(X) = c_{\T} \cdot \frac{62}{63\sqrt{27}\zeta(6)} X^{1/2} + O(X^{1/4}),
\]
where 
\[
c_{\T} = 
\begin{cases} 
243/364 & \text{ if $\T = \II$} \\
81/364& \text{ if $\T = \III$} \\
27/364& \text{ if $\T = \IV$} \\
9/364& \text{ if $\T = \IV^*$} \\
3/364& \text{ if $\T = \III^*$, and} \\
1/364& \text{ if $\T = \II^*$}.
\end{cases}
\]
\end{proof}

We conclude by performing some calculations in Pari/GP as a verification of Theorem \ref{j=0_tors=0}.  For $1 \leq t \leq 192450$, the elliptic curves 
\[
y^2 = x^3 + t
\]
have height $<10^{12}$.   For these values of $t$ we compute the exact number $N_{\lbrace \infty \rbrace,\T}^0(10^{12})$ of elliptic curves that are minimal, have trivial torsion, $j$-invariant 0, and Kodaira-N\'eron type $\T$.  We then compare this to the function
\[
f(X,\T) = c_{\T} \cdot \frac{62\sqrt{X}}{63\sqrt{27}\zeta(6)}
\]
at $X = 10^{12}$.   We find that the calculations match within the predicted error estimate.

\begin{center}
\begin{tabular}{|l|r|r|}
\hline
$\T$ & $N_{\lbrace \infty \rbrace,\T}^0(10^{12})$ & $f(10^{12},\T)$ \\
\hline
$\II$  & 124138 & $\sim 124281.5$ \\
$\III$ & 41331& $\sim 41427.2$ \\
$\IV$ & 13736 & $\sim 13809.2$  \\
$\IV^*$ & 4579& $\sim 4603.1$  \\
$\III^*$ &  1525& $\sim 1534.3$ \\
$\IV^*$ &512 & $\sim 511.4$ \\
\hline
\end{tabular}
\end{center}

\vfill

\section{Tables}

\subsection{Isogeny-Torsion Graphs} \label{iso_tor_subsection}

\begin{center}
\begin{table}[h]
\begin{tabular}{|cc|}
\hline
$T_4^1:$&
${ \xymatrix{
& [2] \atop j=66^3 & \\
[4] \atop j=66^3 \ar@{-}[r]&[2,2] \atop j=1728 \ar@{-}[u] \ar@{-}[r]& [4] \atop j =1728
}}$ \\
\hline
$T_4^2:$ &${ \xymatrix{
& [2] \atop j=66^3 & \\
[4] \atop j=66^3 \ar@{-}[r]&[2,2] \atop j=1728 \ar@{-}[u] \ar@{-}[r]& [2] \atop j =1728
}}$ \\
\hline
$T_4^3:$ & ${ \xymatrix{
& [2] \atop j=66^3 & \\
[2] \atop j=66^3 \ar@{-}[r]&[2,2] \atop j=1728 \ar@{-}[u] \ar@{-}[r]& [2] \atop j =1728
}}$ \\
\hline
$L_2(2):$ & ${ \xymatrix{
[2] \atop j=1728 \ar@{-}[r]& [2] \atop j=1728 }}$ \\
\hline
\end{tabular}
 \label{1728_iso_tor_graphs}
\caption{Isogeny Torsion Graphs for $j=1728$}
\end{table}
\end{center}

\vfill

\newpage

\subsection{Kodaira-N\'eron Types for $j$-invariant 0} 

\begin{center}
\begin{table}[h] 
\begin{tabular}{|cl|}
\hline
$E$ & $\typ_3(E)$ \\
\hline
\href{https://beta.lmfdb.org/EllipticCurve/Q/27/a/3}{27.a3} & $\IV^*$  \\
\href{https://beta.lmfdb.org/EllipticCurve/Q/27/a/4}{27.a4} & $\II$ \\
\href{https://beta.lmfdb.org/EllipticCurve/Q/36/a/3}{36.a3} & $\III^*$ \\
\href{https://beta.lmfdb.org/EllipticCurve/Q/36/a/4}{36.a4} & $\III$ \\
\hline
\end{tabular}
\caption{$\typ_3(E)$ for \href{https://beta.lmfdb.org/EllipticCurve/Q/27/a/3}{27.a3},
\href{https://beta.lmfdb.org/EllipticCurve/Q/27/a/4}{27.a4}, 
\href{https://beta.lmfdb.org/EllipticCurve/Q/36/a/3}{36.a3}, 
\href{https://beta.lmfdb.org/EllipticCurve/Q/36/a/4}{36.a4}} \label{j0_finite_tables}
\end{table} 
\end{center}

\begin{center}
 \begin{table}[h]
 \begin{tabular}{|l|l|} 
 \hline
$\typ_3(E)$ & Conditions \\
 \hline
$\II$ & $b=t^2$,  $v_3(t)=0$, and $t\equiv \pm 2, \pm 4 \pmod{9}$ \\
$\III$ & $b=t^2$, $v_3(t)=0$, and $t\equiv \pm 1 \pmod{9}$  \\
$\IV$ & $b=t^2$, $v_3(t)=1$ \\
$\IV^*$ & $b=t^2$, $v_3(t)=2$  \\
 \hline
\end{tabular}
 \caption{$j$-invariant 0 and $E(\Q)_\tors \simeq \Z/3\Z$}
 \label{j0_Z3_torsion}
 \end{table}
 \end{center}

\begin{center}
 \begin{table}[h]
 \begin{tabular}{|l|l|}
 \hline
$\typ_3(E)$ & Conditions \\
 \hline
 $\III$ & $b$ is a cube, $v_3(b)=0$ \\
 $\III^*$ & $b$ is a cube, $v_3(b) =3$\\
 \hline
\end{tabular}
 \caption{$j$-invariant 0 and $E(\Q)_\tors \simeq \Z/2\Z$}  \label{j0_Z2_torsion}
 \end{table}
 \end{center}
 
\begin{center}
\begin{table}[h]
 \begin{tabular}{|l|l|}
 \hline
$\typ_3(E)$ & Conditions \\
 \hline
 $\II$ & $b \equiv \pm 2, \pm 3,\pm 4 \pmod{9}$\\
 $\II^*$ &$v_3(b) = 5$\\
 $\III$ & $b\equiv \pm 1\pmod{9}$ \\
 $\III^*$ & $v_3(b) = 3$, $b/27 \equiv \pm 1 \pmod{9}$ \\
 $\IV$ & $v_3(b)=2$\\
 $\IV^*$ & $v_3(b)=4$, or\\
 &$v_3(b)=3$ and $b/27 \equiv \pm 2,\pm4 \pmod{9}$ \\
 \hline
\end{tabular}
 \caption{$j$-invariant 0 and $E(\Q)_\tors \simeq \lbrace \infty \rbrace$} \label{j0_trivial_torsion}
 \end{table} 
 \end{center}

\newpage

\subsection{Kodaira-N\'eron Types for $j$-invariant 1728}

\begin{center}
\begin{table}[h] 
\begin{tabular}{|cl|}
\hline
$E$ & $\typ_2(E)$ \\
\hline
\href{https://beta.lmfdb.org/EllipticCurve/Q/32/a/3}{32.a3} & $\III$  \\
\href{https://beta.lmfdb.org/EllipticCurve/Q/34/a/4}{32.a4} & $\I_3^*$ \\
\href{https://beta.lmfdb.org/EllipticCurve/Q/64/a/3}{64.a3} & $\I_2^*$ \\
\href{https://beta.lmfdb.org/EllipticCurve/Q/64/a/4}{64.a4} & $\II$ \\
\hline
\end{tabular}
\caption{$\typ_2(E)$ for \href{https://beta.lmfdb.org/EllipticCurve/Q/32/a/3}{32.a3},
\href{https://beta.lmfdb.org/EllipticCurve/Q/32/a/4}{32.a4}, 
\href{https://beta.lmfdb.org/EllipticCurve/Q/64/a/3}{64.a3}, 
\href{https://beta.lmfdb.org/EllipticCurve/Q/64/a/4}{64.a4}} \label{j1728_finite_tables}
\end{table} 
\end{center}

\begin{center}
 \begin{table}[h]
 \begin{tabular}{|l|l|}
 \hline
$\typ_2(E)$ & Conditions \\
 \hline
 $\I_2^*$ & $v_2(a) = 2$\\
 $\III$ & $v_2(a) = 0$\\ 
  \hline 
\end{tabular}
 \caption{$j$-invariant 1728, $E(\Q)_\tors \simeq \Z/2\Z \times \Z/2\Z$} \label{j1728_(2,2)tors}
 \end{table}
 \end{center}

\begin{center}
\begin{table}[h]
 \begin{tabular}{|l|l|}
 \hline
$\typ_2(E)$ & Conditions \\
 \hline
 $\II$ & $t$ even \\
 $\I_3^*$ & $t$ odd \\
\hline
\end{tabular}
 \caption{$j$-invariant 1728, $E(\Q)_\tors \simeq \Z/2\Z$, Isogeny Class size 4} \label{j1728-Z2tors-size4}
 \end{table}
  \end{center}

\begin{center}
 \begin{table}[h]
 \begin{tabular}{|l|l|}
 \hline
$\typ_2(E)$ & Conditions \\
 \hline
 $\I_2^*$ & $v_2(a)=2$ and $a/4 \equiv 3 \pmod{4}$ \\
$\I_3^*$ & $v_2(a) = 2$ and $a/4 \equiv 1 \pmod{4}$\\
$\II$ & $v_2(a)=0$ and $a \equiv 1 \pmod{4}$\\
$\III$ & $v_2(a)=1$, or \\
 &$v_2(a)=0$ and $a \equiv 3 \pmod{4}$\\
 $\III^*$ & $v_2(a)=3$ \\
\hline
\end{tabular}
 \caption{$j$-invariant 1728, $E(\Q)_\tors \simeq \Z/2\Z$, Isogeny Class size 2} \label{j1728-generic}
 \end{table}
 \end{center}

\end{document}